\documentclass[12pt]{amsart}
\usepackage{graphicx}

\usepackage {amsmath, amssymb, amscd,thmtools}
\usepackage{placeins}
\usepackage[margin=1in]{geometry}
\usepackage{cancel}
\usepackage{graphicx}
\usepackage[hidelinks]{hyperref}
\usepackage{tikz-cd}
\usepackage[utf8]{inputenc}
\usepackage{cases}
\usepackage{multicol}
\usepackage{amsaddr}
\allowdisplaybreaks

\newcommand{\norm}[1]{\|#1\|}

\renewcommand{\L}{\mathcal{L}}

\newcommand{\R}{\mathbb{R}}

\newcommand{\T}{\mathbb{T}}

\renewcommand{\H}{\mathcal{H}}

\renewcommand{\div}{\mbox{div}}

\newcommand{\grad}{\nabla}
\newcommand{\Lap}{\Delta}
\newcommand{\A}{\mbox{\AA}}
\newcommand{\scrA}{\mathcal{A}}
\newcommand{\calO}{\mathcal{O}}
\renewcommand{\P}{\mathbb{P}}

\newcommand{\Lag}{\text{Lag}}
\newcommand{\Herm}{\text{Herm}}

\newcommand{\obar}[1]{\mkern 1.5mu\overline{\mkern-1.5mu#1\mkern-1.5mu}\mkern 1.5mu}

\newcommand{\vep}{\varepsilon}

\DeclareMathOperator{\im}{Im}

\DeclareMathOperator{\spanset}{span}

\newtheorem{theorem}{Theorem}
\newtheorem{assumption}{Assumption}
\newtheorem{corollary}[theorem]{Corollary}
\begin{document}

\title{Gevrey Regularity for a Fluid-Structure Interaction Model}

\author{George Avalos}
\address{University of Nebraska-Lincoln \\
Lincoln, NE 68588, USA \\
gavalos2@unl.edu}
\thanks{Author 1 partially supported by NSF Grant DMS-1948942}
\author{Dylan McKnight}
\address{Colorado Mesa University \\
Grand Junction, CO 81501, USA \\ dmcknight@coloradomesa.edu}
\author{Sara McKnight}
\address{Colorado Mesa University \\
Grand Junction, CO 81501, USA \\ smcknight@coloradomesa.edu}

\begin{abstract}
A result of Gevrey regularity is ascertained for a semigroup which models a fluid-structure interaction problem. In this model, the fluid evolves in a piecewise smooth or convex geometry $\calO$. On a portion of the boundary, a fourth order plate equation is coupled with the fluid through pressure and matching velocities. The key to obtaining the conclusion of Gevrey regularity is an appropriate estimation of the resolvent of the associated $C_0$-semigroup operator. Moreover, a numerical scheme and example is provided which empirically demonstrates smoothing of the fluid-structure semigroup.
\end{abstract}

\maketitle

\null


\null


\section{Introduction}
During the course of communication between the authors of \cite{AC14} and the late Igor Chueshov, the latter posed the conjecture that the $C_0$-semigroup $\{e^{\scrA t}\}_{\mathcal{L}(\H)}$ -- formulated below -- which models the fluid-structure interaction (FSI) \eqref{FSI} below, besides manifesting exponential decay -- see \cite{AB14} -- is actually analytic in an appropriate sector of the complex plane. Generally, couplings of hyperbolic and parabolic dynamics which evolve on distinct geometries and interact strictly accross a boundary interface are not at all analytic or analytic-like; see \cite{AT07}, \cite{BG07}, \cite{BT15}, \cite{CD05} (and references therein).

However, unlike the FSI systems in said references, the PDE components in FSI \eqref{FSI} below do \emph{not} share the same dimensionality: in particular, in \eqref{FSI}, the three dimensional incompressible flow ``immerses" the two dimensional plate dynamics. Accordingly, it would seem reasonable to attempt to acertain the Gevrey regularity of the associated FSI semigroup. Recall that a strongly continuous semigroup $S(t)$ on a Banach space $B$ is of Gevrey class $\delta$ (where $\delta > 1)$ if $S(t)$ is infinitely differentiable on $(t_0,\infty),$ for some $t_0 > 0,$ and for every compact $K \in (t_0,\infty)$ and for any $\theta > 0,$ there exists a constant $C > 0$ such that
$$\norm{S^{(n)}(t)}_{\mathcal{L}(B)} \leq C\theta^n (n!)^\delta,$$ for all $k \in K$ and $n \in \mathbb{N}$ (see \cite{T89}, p. 143). Roughly speaking, Gevrey regularity for a strongly continuous semigroup is a state which lies between differentiability and analyticity.

Our main result of Gevrey regularity, in Theorem \ref{thm:mainResult}, is similar to that obtained in \cite{BT22-2}, \cite{BT22-1}, which generally deal with 2D-1D FSI with moving boundary interface. Therein, the necessary Gevrey (resolvent) estimates for their fixed domain models are obtained for two dimensional fluids which are coupled to one dimensional beams, on specified geometrical configurations. On such special geometries, the authors in \cite{BT19} are able to clearly characterize the FSI generator resolvent, and subsequently estimate it.

In contrast, the present work deals with 3D-2D FSI on rough, convex domains -- see Assumption 1 below. Moreover, here there is no clear representation and estimation of the associated FSI resolvent operator (there cannot be, for arbitrary domains). Instead, our short proof uses in an essential way the underlying dissipation which arises from the fluid PDE component of the FSI dynamics.

\section{PDE Model}
Let $\mathcal{O}\subset\R^3$ be a bounded domain such that $\partial\mathcal{O}=\obar{\Omega}\cup\obar{S}$, with $\Omega\cap S=\emptyset$ and $\partial\Omega=\partial S$. Further, the boundary segement $\Omega\subseteq\{(x_1,x_2,0):x_1,x_2\in\R\}$ is flat, and surface $S$ is contained in $\{(x_1,x_2,x_3) \in \R^3: x_3 \leq 0\}.$ The outward unit vector on $\partial\calO$ is denoted $n(x)$; in particular, $n|_\Omega = \langle 0,0,1 \rangle.$ An example of such a geometry is in Figure \ref{geometry1}.

    \begin{center}
    \begin{figure}[htp!]
    \includegraphics[width=6cm]{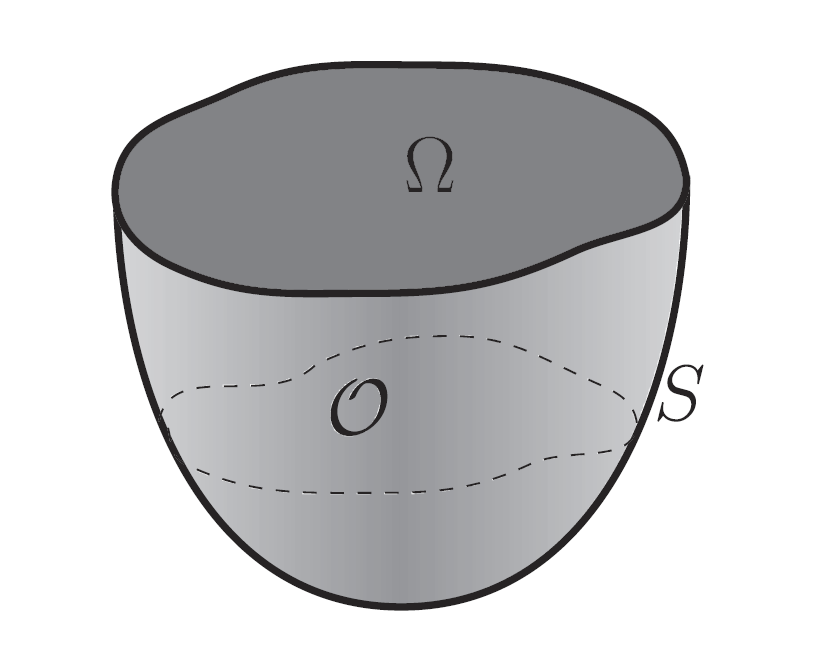}
    \caption{Flow Structure Geometry}
    \label{geometry1}
    \end{figure}
    \end{center}

\FloatBarrier    
Throughout, we will impose the following assumptions on the geometry (see \cite{CR13}):
\begin{assumption}
The pair $\{\calO,\Omega\}$ satisfies one of the following:
\begin{enumerate}
    \item[(G.1)] $\calO$ is convex with wedge angles at most $\frac{2\pi}{3}$, $\Omega$ has smooth boundary, and $S$ is a piecewise smooth surface.
    \item[(G.2)] $\calO$ is a convex polyhedron with angles at most $\frac{2\pi}{3}$, implying $\Omega$ is a convex polygon with angles at most $\frac{2\pi}{3}$.
\end{enumerate}
\end{assumption}
The geometry in Figure \ref{geometry1} satisfies (G.1), while the unit cube in $\R^3$, with $\Omega=\{(x_1,x_2,1):x_1,x_2\in(0,1)\}$, $S=\partial(0,1)^3\setminus\Omega$ is an example of a geometry satisfying $(G.2)$. As established in \cite{AB14}, such assumptions on the geometry allow us to recover the results of \cite{CR13}, in part via the regularity results in \cite{D89}. 

\subsection{The Underlying Abstract Framework}
On this geometry $\{\calO,\Omega\},$ we will consider the following fluid-structure interaction (FSI):

\begin{align}
    \begin{cases}
     w_{tt}+\A w = p|_{\Omega} \text{ in } \Omega\times(0,T)\\
    w=\frac{\partial w}{\partial n}=0 \text{ on }\partial\Omega\\
    u_t-\Lap u +\grad p = 0 \text{ in }\mathcal{O}\times(0,T)\\
    \div(u) = 0 \text{ in }\mathcal{O}\times(0,T)\\
    u=0\text{ on }S,~u=\begin{bmatrix}0,0,w_t\end{bmatrix}\text{ on }\Omega,\\
    [w(0),w_t(0),u(0)] = [w_1^0,w_2^0,u^0]\in\H,
    \end{cases}\label{FSI}
\end{align}

We start by writing the biharmonic operator with clamped boundary conditions:
\begin{center}
\begin{tabular}{l l}
    $\A:=\Lap^2$, &  $D(\A)=H^4(\Omega)\cap H_0^2(\Omega)$,\\
\end{tabular}
\end{center}
As defined, $\A:D(\A)\to L^2(\Omega)$ is a positive definite, self-adjoint operator, and so $\A^s$ is well-defined for all $s\in\R$. Note that in making this conclusion, we are implicitly using the higher regularity results of \cite{BR80} (see theorem 2, p. 563 therein) and the geometric Assumption 1 to conclude that $\A$ is an isomorphism of $L^2(\Omega)$ onto $H^4(\Omega)\cap H_0^2(\Omega).$

Moreover, we define the following Robin maps $R\in\L(H^{-1/2}(\Omega), H^{1}(\calO)),$ $\tilde{R}\in\L(H^{-1/2}(S), H^1(\calO))$:
\begin{align*}
    R(g)=f \mbox{ iff } \begin{cases}\Lap f=0 \mbox{ in } \mathcal{O}\\ \frac{\partial f}{\partial n}+f=g\mbox{ on }\Omega\\\frac{\partial f}{\partial n}=0\mbox{ on }S\end{cases}\\
\tilde{R}(g)=f \mbox{ iff } \begin{cases}\Lap f=0 \mbox{ in } \mathcal{O}\\ \frac{\partial f}{\partial n}+f=0\mbox{ on }\Omega\\\frac{\partial f}{\partial n}=g\mbox{ on }S\end{cases}
\end{align*}

\noindent Subsequently, we set
\begin{align}
    G_{1}(w_1)&=R(\A w_1), \label{g1def}\\
    G_{2}(u)&=R(\Lap u^{(3)}|_\Omega)+\tilde{R}(\Lap u\cdot n|_S).\label{g2def}
\end{align}
The associated finite energy space $\H$ is given by:
\begin{align}
    \mathcal{H} &= \left\{\begin{bmatrix}w_1\\w_2\\u\end{bmatrix}\in [H_0^2(\Omega)\cap\hat{L}^2(\Omega)]\times \hat{L}^2(\Omega)\times\mathcal{H}_f:u\cdot n|_S=0,~u=\begin{bmatrix}0\\0\\ w_2\end{bmatrix}\text{ on }\Omega\right\}.
\end{align}
 Here, $\mathcal{H}_f$ is
\begin{align*}
    \mathcal{H}_f:=\{u\in L^2(\mathcal{O}):\div(u)=0,~u\cdot n|_S=0\},
\end{align*}
and $$\hat{L}^2(\Omega) = \left\{g\in L^2(\Omega) : \int_\Omega g \, d\Omega = 0\right\}.$$
Throughout, we will use the following notation for inner products and norms:
\begin{align}
    \norm{w}_\Omega&=\norm{w}_{L^2(\Omega)},~\langle w,\tilde{w}\rangle_{\Omega}=(w,\tilde{w})_{L^2(\Omega)},\\
    \norm{u}_\calO&=\norm{u}_{L^2(\calO)},~~(u,\tilde{u})_{\calO}=(u,\tilde{u})_{L^2(\calO)}.
\end{align}
Given the characterization $D\left(\A^{\frac{1}{2}}\right)\cong H_0^2(\Omega)$, by Theorem 8.1 of \cite{G67}, we equip $\H$ with the inner product
\begin{align*}
    \langle \Phi, \tilde{\Phi}\rangle_{\H}=\left\langle\begin{bmatrix}w_1\\ w_2\\u\end{bmatrix}, \begin{bmatrix} \tilde{w}_1\\\tilde{w}_2\\\tilde{u}\end{bmatrix}\right\rangle_{\H}=\left\langle \A^{\frac{1}{2}}w_1, \A^{\frac{1}{2}}\tilde{w}_1\right\rangle_{\Omega}+\langle  w_2,  \tilde{w}_2\rangle_{\Omega} + (u, \tilde{u})_{\calO}.
\end{align*}

\noindent Following the procedure laid out in \cite{AC14}, it can be shown that the pressure $p$ of the FSI can be identified as the solution of the following elliptic boundary value problem:
\begin{align}
\begin{cases}
    \Lap p = 0\text{ in }\mathcal{O}\\
    \frac{\partial p}{\partial n} = \Lap u\cdot n|_\Omega \text{ on }S\\
    \frac{\partial p}{\partial n}+p=\A w+(\Lap u)^{(3)}|_\Omega\text{ on }\Omega
    \end{cases}\label{PBVP}
\end{align}

\noindent Given that $p$ solves the above boundary value problem, one has that 
\begin{equation}
p(t)=G_{1}(w(t))+G_{2}(u(t))\label{pressure}
\end{equation}
pointwise in time. Thus, FSI \eqref{FSI} can be written as the abstract evolution equation $\frac{d\Phi}{dt}=\mathcal{A} \Phi$, where
\begin{align}
    \mathcal{A} = \begin{bmatrix} 0&I&0\\ -\A+G_{1}|_{\Omega}&0&G_{ 2}|_\Omega\\
    -\grad G_{1}&0&\Lap-\grad G_{2}
    \end{bmatrix},~~\Phi=\begin{bmatrix}w\\w_t\\u\end{bmatrix},~~\text{and}~~ \Phi(0) = \Phi_0\label{GEN}.
\end{align}
Semigroup well-posedness of this abstract evolution equation was established in \cite{AC14}. Specifically, $\scrA$ is the infinitesimal generator of a $C_0-$semigroup of contractions on finite energy space $\H$. 
The domain of the generator $\scrA$ is given by
\begin{align}
    D(\mathcal{A})&=\left\{\begin{bmatrix}w_1\\w_2\\u\end{bmatrix}\in\mathcal{H}:w_1\in H^4(\Omega),~w_2\in H_0^2(\Omega), ~u\in H^2(\mathcal{O}), ~u=0\text{ on }S, ~u = [0,0,w_2]\text{ on }\Omega\right\}.\label{domain}
\end{align}

\begin{theorem}{(see Theorem 1.1 of \cite{AC14} and Remark 3.2 of \cite{AB14})} Under geometrical Assumption 1, $\scrA:D(\scrA)\subset\H \to \H,$ as defined in \eqref{GEN}-\eqref{domain}, generates a $C_0$-semigroup of contractions on $\H.$
\end{theorem}



\section{Gevrey Class}
\noindent By Gevrey class, we mean the following estimate on the resolvent of $\scrA$ along some vertical line in the right half complex plane.

\begin{theorem}[Theorem 4 of \cite{T89}]
Let $S(t)$ be a $C_0$-semigroup generated by $\mathcal{A}$ satisfying $\norm{S(t)}\leq M e^{\omega t}$. Suppose that for some $\mu\geq\omega$ and $\alpha\in(0,1]$,
\begin{align}
    \limsup_{|\beta|\to\infty}|\beta|^\alpha\norm{R(\mu+i\beta;\mathcal{A})}=C<\infty.
\end{align}
then $S(t)$ is of Gevrey class $\delta$ for $t>0$ for each $\delta>\alpha^{-1}$. 
\end{theorem}
\noindent For  $C_0-$semigroups of contractions, one has that $M=1$ and $\omega=0$. Hence, we will take $\mu=0$ (c.f. Theorem 1.1 of \cite{CT90}). Intuitively, one may think of Gevrey class as an intermediary between differentiability and analyticity.

The main result to be proved here is as follows:



\begin{theorem}\label{thm:mainResult}
    There exists a constant $C > 0$ such that
    \begin{equation}\label{gevbound}
    \|R(i\beta;\scrA)\|_{op} \leq \frac{C}{|\beta|^\alpha},
    \end{equation}
    for all $\beta \in \R$ which satisfies $|\beta| \geq \beta_0 > 0$ ($\beta_0$ fixed). Accordingly, the FSI semigroup $\{e^{\scrA t}\}_{t\geq 0}$ is of Gevrey class $\delta = 2+\rho$, for $t>0,$ where $\rho > 0$ is arbitrarily small.
\end{theorem}

\section{Gevrey Class Estimates of the Generator}

\begin{proof}[Proof of Theorem \ref{thm:mainResult}]
Given $\scrA: D(\scrA)\subset \H \to \H$ of \eqref{GEN} - \eqref{domain}, we consider the resolvent equation
\begin{align}
    (i\beta I-\scrA)\Phi=\Phi^*.\label{reseqn}
\end{align}
where dependent variable $\Phi = [w_1,w_2,u]^T \in D(\scrA)$ and data $\Phi^* = [w_1^*, w_2^*, u^*]^T \in \H.$ Also, fixed (``frequency domain") parameter $\beta$ satisfies $|\beta| > \beta_0,$ where $\beta_0 \geq 1.$

In view of \eqref{gevbound}, we must obtain a suitable estimate for the solution variable $\Phi.$ Taking the inner product of both sides of \eqref{reseqn} with respect to $\Phi,$ we have
\begin{equation}\label{resinnerproduct}
    ((i\beta I - \scrA)\Phi, \Phi)_\H = (\Phi^*,\Phi)_\H.
\end{equation}
Now, using the definition of $\scrA$ in \eqref{GEN}, we have
\begin{align*}
    (-\scrA\Phi,\Phi)_\H = & -\langle \A^{\frac{1}{2}}w_2, \A^{\frac{1}{2}}w_1\rangle_\Omega + \langle \A w_1,w_2\rangle_\Omega - \langle p|_\Omega, w_2\rangle_\Omega- (\Delta u,u)_\calO + (\grad p,u)_\calO.
\end{align*}
Here, the pressure variable $p$ is given viz.
$$p = G_1(w_1) + G_2(u),$$ as in \eqref{pressure} (and \eqref{g1def}-\eqref{g2def}).

Integrating by parts and using the matching velocity boundary condition $u|_\Omega = [0,0,w_2]$ (as well as $u|_{\partial\calO\setminus\Omega} = 0),$ we have
\begin{align*}
    -(\scrA\Phi,\Phi)_\H = 2i\im\left\langle \A^{\frac{1}{2}}w_1, \A^{\frac{1}{2}}w_2\right\rangle_\Omega - \|\grad u\|_\calO^2.
\end{align*}
Applying this relation to the left hand side of \eqref{resinnerproduct}, we have now
\begin{align}
    i\beta\norm{\Phi}^2_\H+2i\im\langle\A w_1,w_2\rangle_\Omega+\norm{\grad u}_\calO^2=\langle \Phi^*,\Phi\rangle_\H.\label{maineqn}
\end{align}
Taking real and imaginary parts of \eqref{maineqn} and applying the triangle inequality gives the following inequalities:
\begin{align}
    \norm{\Phi}_\H^2&\leq\frac{1}{|\beta|}|\langle \Phi^*,\Phi\rangle_\H|+\frac{2}{|\beta|}|\langle\A w_1,w_2\rangle_\Omega|,\label{mainineq}\\
    \norm{\grad u}_\calO^2&\leq|\langle \Phi^*,\Phi\rangle_\H|.\label{fluidineq}
\end{align}

The main task is to bound the second term on the right hand side of \eqref{mainineq}. At this point, we recall the characterization of the domains of fractional powers of $\A:D(\A)\subset L^2(\Omega) \to L^2(\Omega)$ in \cite{G67}. Namely,
\begin{equation}\label{fractionalDomain}
    D(\A^\alpha) = \begin{cases} 
    H^{4\alpha}(\Omega), & \text{ for } 0 \leq \alpha < \frac{1}{8} \\
    H_0^{4\alpha}(\Omega), & \text{ for } \alpha \in \left(\frac{1}{8},\frac{3}{8}\right) \cup \left(\frac{3}{8},\frac{1}{2}\right] \\
    H^{4\alpha}(\Omega)\cap H_0^2(\Omega), & \text{ for } \frac{1}{2} \leq \alpha \leq 1.
    \end{cases}
\end{equation}
Recalling the definition of the FSI generator domain $D(\scrA)$ from \eqref{domain}, we then transfer $\A^{\frac{1}{8}-\vep}$ from $w_1$ onto $w_2$, where $\vep>0$ is small, and interpolate.
\begin{align}
    \frac{2}{|\beta|}|\langle\A w_1,w_2\rangle_\Omega|&=\frac{2}{|\beta|}|\langle\A^{\frac{7}{8}+\vep}w_1, \A^{\frac{1}{8}-\vep}w_2\rangle_\Omega| \\
    &\leq\frac{2}{|\beta|}\norm{\A^{\frac{7}{8}+\vep}w_1}_\Omega\norm{\A^{\frac{1}{8}-\vep}w_2}_\Omega \\
    &\leq \frac{2}{|\beta|}\norm{\A w_1}_\Omega^{\frac{3}{4}+2\vep}\norm{\A^{\frac{1}{2}}w_1}_\Omega^{\frac{1}{4}-2\vep}\norm{\A^{\frac{1}{8}-\vep}w_2}_\Omega. \label{interpIneq}
\end{align}
Next, we note that $D(\A^{\frac{1}{8}-\vep})=H^{\frac{1}{2}-4\vep}(\Omega)$ from \eqref{fractionalDomain}. Applying this characterization to the last term of inequality \eqref{interpIneq} and invoking the Sobolev Trace Theorem gives
\begin{align*}
     \frac{2}{|\beta|}\norm{\A w_1}_\Omega^{\frac{3}{4}+2\vep}\norm{\A^{\frac{1}{2}}w_1}_\Omega^{\frac{1}{4}-2\vep}\norm{\A^{\frac{1}{8}-\vep}w_2}_\Omega&\leq \frac{C}{|\beta|}\norm{\A w_1}_\Omega^{\frac{3}{4}+2\vep}\norm{\A^{\frac{1}{2}}w_1}_\Omega^{\frac{1}{4}-2\vep}\norm{w_2}_{H^{\frac{1}{2}-4\vep}(\Omega)}\\
     &\leq \frac{C}{|\beta|}\norm{\A w_1}_\Omega^{\frac{3}{4}+2\vep}\norm{\A^{\frac{1}{2}}w_1}_\Omega^{\frac{1}{4}-2\vep}\norm{w_2}_{H^{\frac{1}{2}}(\Omega)}\\
     &\leq \frac{C}{|\beta|}\norm{\A w_1}_\Omega^{\frac{3}{4}+2\vep}\norm{\A^{\frac{1}{2}}w_1}_\Omega^{\frac{1}{4}-2\vep}\norm{u}_{H^1(\calO)}\\
     &\leq \frac{C}{|\beta|}\norm{\A w_1}_\Omega^{\frac{3}{4}+2\vep}\norm{\A^{\frac{1}{2}}w_1}_\Omega^{\frac{1}{4}-2\vep}\norm{\grad u}_{\calO}.
\end{align*}
In obtaining this estimate, we have also invoked the matching of velocities on $\Omega$ and Poincare's inequality. We next turn our attention to the two terms involving $w_1$. Majorizing by the measurements $\norm{\scrA\Phi}_\H$ and $\norm{\Phi}_\H$, then applying \eqref{reseqn} and \eqref{fluidineq} gives
\begin{align*}
    \frac{C}{|\beta|}\norm{\A w_1}_\Omega^{\frac{3}{4}+2\vep}\norm{\A^{\frac{1}{2}}w_1}_\Omega^{\frac{1}{4}-2\vep}\norm{\grad u}_{\calO}&\leq \frac{C}{|\beta|}\norm{\scrA \Phi}_\H^{\frac{3}{4}+2\vep}\norm{\Phi}_\H^{\frac{1}{4}-2\vep}\norm{\grad u}_{\calO}\\
    &\leq \frac{C}{|\beta|}\norm{i\beta \Phi-\Phi^*}_\H^{\frac{3}{4}+2\vep}\norm{\Phi}_\H^{\frac{1}{4}-2\vep}\norm{\Phi}_\H^{\frac{1}{2}}\norm{\Phi^*}_\H^{\frac{1}{2}}\\
    &\leq \frac{C}{|\beta|}\norm{\Phi}_\H^{\frac{1}{4}-2\vep}\norm{\Phi}_\H^{\frac{1}{2}}\norm{\Phi^*}_\H^{\frac{1}{2}}(|\beta|^{\frac{3}{4}+2\vep}\norm{ \Phi}_\H^{\frac{3}{4}+2\vep}+\norm{\Phi^*}_\H^{\frac{3}{4}+2\vep})\\
    &= \frac{C}{|\beta|}(|\beta|^{\frac{3}{4}+2\vep}\norm{\Phi}_\H^{\frac{3}{2}}\norm{\Phi^*}_\H^{\frac{1}{2}}+\norm{\Phi}_\H^{\frac{3}{4}-2\vep}\norm{\Phi^*}_\H^{\frac{5}{4}+2\vep}).
\end{align*}
Finally, take $\vep_1>0$. We apply Young's inequality with $\vep_1$ (Appendix B.2.c of \cite{E98}) to the two terms on the right hand side to get
\begin{align}
    \frac{C}{|\beta|}|\beta|^{\frac{3}{4}+2\vep}\norm{\Phi}_\H^{\frac{3}{2}}\norm{\Phi^*}_\H^{\frac{1}{2}}&=\frac{C}{|\beta|^{\frac{1}{4}-2\vep}}\norm{\Phi}_\H^{\frac{3}{2}}\norm{\Phi^*}_\H^{\frac{1}{2}}\leq C\vep_1\norm{\Phi}_\H^2+\frac{C(\vep_1)}{|\beta|^{1-8\vep}}\norm{\Phi^*}^2_\H,\label{ineq1}\\
    \frac{C}{|\beta|}\norm{\Phi}_\H^{\frac{3}{4}-2\vep}\norm{\Phi^*}_\H^{\frac{5}{4}+2\vep}&\leq C\vep_1\norm{\Phi}_\H^2+\frac{C(\vep_1)}{|\beta|^{\frac{8}{5+8\vep}}}\norm{\Phi^*}_\H^2. \label{ineq2}
\end{align}
Next, note that $1-8\vep<\frac{2}{\frac{5}{4}+2\vep}$, so $|\beta|^{1-8\vep}<|\beta|^\frac{2}{\frac{5}{4}+2\vep}$. Applying this to \eqref{ineq2} combining with \eqref{ineq1}, and rescaling $\vep > 0$ gives the final estimate
\begin{align}
    \frac{2}{|\beta|}|\langle\A w_1, w_2\rangle_\Omega|\leq C\vep_1\norm{\Phi}_\H^2+\frac{C(\vep_1)}{|\beta|^{1-8\vep}}\norm{\Phi^*}^2.\label{pltineq}
\end{align}
Returning to \eqref{mainineq}, we apply Young's inequality with $\vep_1$ to the first term of the right hand side, and \eqref{pltineq} to the second term of the right hand side to obtain
\begin{align}
    \norm{\Phi}_\H^2&\leq \vep_1\norm{\Phi}^2_\H+\frac{C(\vep_1)}{|\beta|^2}\norm{\Phi^*}^2_\H+C\vep_1\norm{\Phi}_\H^2+\frac{C(\vep_1)}{|\beta|^{1-8\vep}}\norm{\Phi^*}_\H^2\\
    &\leq C\vep_1\norm{\Phi}_\H^2+\frac{C(\vep_1)}{|\beta|^{1-8\vep}}\norm{\Phi^*}^2_\H.
\end{align}
Take $\vep_1$ sufficiently small so that $1-C\vep_1\in(0,1)$. Rearranging and taking the square root yields
\begin{align}
    \norm{\Phi}_\H\leq\frac{C(\vep_1)}{(1-C\vep_1)|\beta|^{\frac{1}{2}-4\vep}}\norm{\Phi^*}_\H.
\end{align}
That is, we infer the resolvent estimate
\begin{align}
    \norm{R(i\beta;\scrA)}_{op}\leq \frac{C(\vep_1)}{(1-C\vep_1)|\beta|^{\frac{1}{2}-4\vep}}.
\end{align}
To conclude: The FSI semigroup S(t) = exp(tA) is of Gevrey class $\delta = \frac{2}{1-8\vep},$ or with $\rho = \frac{16\vep}{1-8\vep}$ where $\vep$ is arbitrarily small, $\delta = 2+\rho,$ and $t \geq 0.$ \qed
\end{proof}

One should note that this result arises in absence of any damping on the plate. For $\eta\in(0,1]$ one may ask how the presence of a damping term on the plate affects the overall regularity of the system. This is encapsulated by the following result.
\begin{corollary}
In the fluid structure system \eqref{FSI}, replace the undamped plate equation with the damped plate equation
\begin{align*}
    w_{tt}+\A w+\A^{\frac{\eta}{2}}w_t=p|_\Omega.
\end{align*}
The generator of this $C_0$-semigroup of contractions, $\scrA_\eta$, is such that $D(\scrA_\eta)=D(\scrA)$ for all $\eta\in(0,1]$, and in addition, the semigroup is Gevrey class $\delta\geq\frac{1}{\eta}$ for $t\geq0$. In particular, for $\eta=1$, the semigroup is analytic.
\end{corollary}

\section{Numerical Results}
\noindent With Theorem \ref{thm:mainResult} in hand, one would expect to see empirical evidence of ``infinite speed of propagation'' in given numerical simulations. That is, rough appropriate initial data is ``smoothed out" as time evolves. Below, we supply a numerical scheme and example which reinforces this assertion. The issue of convergence of the finite dimensional system to the fully infinite dimensional system \eqref{FSI} will be addressed in a future work. Numerical results which suggest this convergence are supplied in \cite{M24}.

We take $\calO = (0,1)^2\subset \R^2$ with $\Omega = \{(x,1):x\in(0,1)\}$ and $S=\partial\calO\setminus\Omega$. Note that this geometry satisfies geometric assumption (G.2). With respect to this geometry $\{\calO,\Omega\},$ the FSI system becomes
\begin{equation} \label{eqn:2dFSI}
    \begin{cases} w_{tt} + \frac{d^4 w}{dx^4} = p|_\Omega \text{ in } \Omega \times (0,\infty) \\
    w = w_x = 0 \text{ on } x = 0 \text{ and } x = 1, t > 0 \\
    u_t - \Delta u + \grad p = 0 \text{ in } \calO\times (0,\infty) \\
    \text{div}(u) = 0 \text{ in } \calO \times (0,\infty) \\
    u = 0 \text{ on } S, u = [0,0,w_t] \text{ on } \Omega \\
    [w(0),w_t(0),u(0)] = [w_1^0, w_2^0, u^0] \in \H.
    \end{cases}
\end{equation}

To deal with geometries satisfying geometric assumption (G.1) via the finite element method (FEM), one could in principle appeal to isoparametric elements (see \cite{AB84}). Here, take $\T$ to be a triangulation of the geometry $\calO$, which consists of $L_f$ triangular elements with $M$ nodes. Relative to the triangulation $\T$, we invoke $\P_2/\P_1$ Taylor-Hood conforming bases, with respect to the fluid and pressure solution components of the FSI (see p. 195 of \cite{AB84}). We thusly define $\{\phi_1,\phi_2,...,\phi_{M}\}$ and $\{\psi_1,\psi_2,...,\psi_{M_p}\}$ to be the fluid and pressure bases, respectively. This basis is appropriately conforming.

As the space of finite energy for the plate component is $H_0^2(\Omega)$, we impose an appropriate $H^2$-conforming basis. In one spatial dimension, one can appeal to standard Hermite elements. In particular we employ the following quintic Hermite element as our standard element: 

\begin{figure}[htbp!]
\begin{center}
\hspace{-0.5in}
\begin{tikzpicture}
\draw[black, thick] (0,0) -- (8,0);

\filldraw[black] (0,0) circle (2pt) ;
\draw[black, thick] (0,0) circle (5pt);
\draw  node[below] at (0,-0.15) {$x_{i_1^\ell}$};
\draw node[above] at (0,0.15) {$\theta_5$};
\draw node[above] at (0,0.65) {$\theta_1$};

\filldraw[black] (8,0) circle (2pt) ;
\draw[black, thick] (8,0) circle (5pt);
\draw  node[below] at (8,-0.15) {$x_{i_2^\ell}$};
\draw node[above] at (8,0.15) {$\theta_6$};
\draw node[above] at (8,0.65) {$\theta_2$};

\filldraw[black] (8/3,0) circle (2pt) ;
\draw  node[below] at (8/3,-0.15) {$x_{i_3^\ell}$};
\draw node[above] at (8/3,0.65) {$\theta_3$};

\filldraw[black] (16/3,0) circle (2pt) ;
\draw  node[below] at (16/3,-0.15) {$x_{i_4^\ell}$};
\draw node[above] at (16/3,0.65) {$\theta_4$};

\draw node[above] at (-1,0.2) {$\partial_x w:$};
\draw node[above] at (-1,0.7) {$w:$};
\end{tikzpicture}
\end{center}
\end{figure}
The above equidistant standard element is appropriate for the computations at hand, yielding mass and plate stiffness matrices which are sufficiently conditioned \cite{S06}. We thus define the Hermite basis $\{\theta_1,\theta_2,...,\theta_{DOF_s}\}$. Denoting by $\tilde{M}$ and $L_s$ the number of plate nodes and elements respectively, one finds by counting that $DOF_s = \tilde{M} + \frac{\tilde{M}+2}{3}$. In our plate basis, the global node indices $\{1,...,\tilde{M}\}$ correspond to the Lagrange interpolation points, and the global node indices $\{\tilde{M}+1,...,DOF_s\}$ correspond to the points where we additionally enforce Hermite interpolation. In extending these dynamics to a 3d-2d coupled system, one would need to employ Argyris elements for the plate, as the usual Hermite elements are not $H^2$ conforming in dimensions higher than 1 (see p. 255 of \cite{S06}).

We these bases in hand, we define the following matrices:
\begin{align}
    \text{(Fluid Mass Matrix)}~M_f&:=\begin{bmatrix}(\phi_1,\phi_1)_\calO&\hdots&(\phi_1,\phi_{M})_\calO\\
    \vdots&\ddots&\vdots\\
    (\phi_{M},\phi_1)_\calO&\hdots&(\phi_{M},\phi_{M})_\calO
    \end{bmatrix}\label{matrix1}\\ 
    \text{(Fluid Stiffness Matrix)}~K_f&:=\begin{bmatrix}(\grad\phi_1,\grad\phi_1)_\calO&\hdots&(\grad\phi_1,\grad\phi_{M})_\calO\\
    \vdots&\ddots&\vdots\\
    (\grad\phi_{M},\grad\phi_1)_\calO&\hdots&(\grad\phi_{M},\grad\phi_{M})_\calO
    \end{bmatrix}\\
    \text{(Divergence Component 1)}~B_x&:=\begin{bmatrix}(\partial_x\phi_1,\psi_1)_\calO&\hdots&(\partial_x\phi_1,\psi_{M_p})_\calO\\
    \vdots&\ddots&\vdots\\
    (\partial_x\phi_{M},\psi_1)_\calO&\hdots&(\partial_x\phi_{M},\psi_{M_p})_\calO
    \end{bmatrix}\\
    \text{(Divergence Component 2)}~B_y&:=\begin{bmatrix}(\partial_y\phi_1,\psi_1)_\calO&\hdots&(\partial_y\phi_1,\psi_{M_p})_\calO\\
    \vdots&\ddots&\vdots\\
    (\partial_y\phi_{M},\psi_1)_\calO&\hdots&(\partial_y\phi_{M},\psi_{M_p})_\calO
    \end{bmatrix}\\
    \text{(Plate Mass Matrix)}~M_s&:=\begin{bmatrix}\langle\theta_1,\theta_1\rangle_\Omega&\hdots&\langle\theta_1,\theta_{\tilde{M}}\rangle_\Omega\\
    \vdots&\ddots&\vdots\\
    \langle\theta_{\tilde{M}},\theta_1\rangle_\Omega&\hdots&\langle\theta_{\tilde{M}},\theta_{\tilde{M}}\rangle_\Omega
    \end{bmatrix}\\
    \text{(Plate ``S'' Matrix)}~S&:=\begin{bmatrix}\langle\theta_1'',\theta_1''\rangle_\Omega&\hdots&\langle\theta_1'',\theta_{\tilde{M}}''\rangle_\Omega\\
    \vdots&\ddots&\vdots\\
    \langle\theta_{\tilde{M}}'',\theta_1''\rangle_\Omega&\hdots&\langle\theta_{\tilde{M}}'',\theta_{\tilde{M}}''\rangle_\Omega
    \end{bmatrix}\label{matrix6}
\end{align}
From \cite{AC14}, recall that the pressure $p$ has two distinct components, a mean free part, and a constant part. We denote this decomposition as
\begin{align}
    p(t) = q(t) + \tilde{c},\label{pres_decomp}
\end{align}
where $q\in \hat{L}^2(\calO)$ and $\tilde{c}\in \R$. We now denote $\{[u_1,u_2,w_1,w_2],p\}$ to be the (finite dimensional) approximation to (infinite dimensional) FSI solution $\{[\vec{u},w,w_t],p\}$ of \eqref{FSI}, which solves the approximating semi-discretized system of ordinary differential equations:
    \begin{align}
    \begin{cases}
        M_f u_1' + K_f u_1 + B_x q = F_{1f},\\
        M_f u_2' + K_f u_2 + B_y q = F_{2f},\\
        B_x^T u_1 + B_y^T u_2 = 0,\\
        M_s w_2' + S w_1 - \tilde{C}E_s^T = F_s,\\
        M_s w_1' = M_s w_2,\\
        E_s w_2 = 0.
    \end{cases}\label{discFSI}
    \end{align}
That is, in the above system we take $u_1$ to mean the solution (c.f. FSI \eqref{eqn:2dFSI}) $u_{1I}\in V_M=\spanset\{\phi_1,\phi_2,...,\phi_M\}\subset H_0^1(\calO)$ and likewise for the other solution variables. The quantity $\tilde{C}=[\langle \tilde{c},\theta_1\rangle_\Omega,...,\langle \tilde{c},\theta_{DOF_s}\rangle_\Omega]^T$ in \ref{discFSI} is the Babuska-Brezzi pressure which arises from the constraint that $w_2$ is mean free as 
\begin{align*}
\int_{\Omega}w_2~d\Omega=\int_\calO\div([u_1,u_2])~d\calO=0.
\end{align*}

The matrices \eqref{matrix1}-\eqref{matrix6} handle the spatial discretization aspect of the system. For temporal discretization, we specify a time step $\Delta t$ and step using Backward Euler:

    \begin{align}
    \begin{cases}
        M_f\alpha^{N+1}_1 + \Delta t K\alpha_1^{N+1} + \Delta t B_x\beta^{N+1} = M_f \alpha_1^N + \Delta t F_{1f}^{N+1},\\
        M_f\alpha^{N+1}_2 + \Delta t K\alpha_1^{N+2} + \Delta t B_y\beta^{N+1} = M_f \alpha_2^N + \Delta t F_{2f}^{N+1},\\
        B_x^T\alpha_1^{N+1} + B_y^T\alpha_2^{N+1} = 0,\\
        M_s\omega_2^{N+1} + \Delta t S\omega_1^{N+1} - \Delta t E_s^T\tilde{c}^{N+1} = M_s\omega_2^N + \Delta t F_s^{N+1},\\
        \omega_1^{N+1} - \Delta t \omega_2^{N+1} = \omega_1^N,\\
        E_s\omega_1^{N+1} = 0.
    \end{cases}\label{fullydiscFSI}
    \end{align}

   To capture the essential boundary conditions of \eqref{FSI}, we have that the degrees of freedom satisfy
    \begin{align}
        \begin{cases}
            \alpha|_S=0,\\
            \alpha|_\Omega = \omega_2,\\
            \omega_1^{\Lag}|_{\partial\Omega} = \omega_2^\Lag|_{\partial\Omega} = 0,\\
            \omega_1^\Herm|_{\partial\Omega} = \omega_2^{\Herm}|_{\partial\Omega} = 0.
        \end{cases}\label{discBCs}
    \end{align}

We require additionally that each of the boundary conditions \eqref{discBCs} is satisfied at each time step $N+1$. In particular, to retain the interesting dynamics between the fluid and plate, we require
\begin{align}
    u_{2, I}^{N+1}(N_i) = w^{N+1}_{2, I}(N_i)\label{velocitymatching}
\end{align}
for each node $N_i$ on the fluid boundary $\Omega$. Note that this constitutes a variational crime -- see \cite{SF73} -- in that the $y$ component of fluid velocity only equals the plate velocity at the nodes, not on the entire boundary $\Omega$. For fine meshes, such an approximation is sufficient (p. 186 of \cite{AB84}). 

Below are results for the smoothing of rough initial data, as pictured in Figures \ref{fig:plate-init} - \ref{fig:fluid-final}. Note that at the terminal time $T = 0.001,$ the fluid FEM approximation is apparently at the zero state. Since solutions of the FSI \eqref{FSI} are known to decay exponentially for given finite energy initial data -- see \cite{AB15} -- this quiescence is not unexpected.

\begin{figure}[h!]
\centering
\begin{multicols}{2}
\includegraphics[width=0.85\linewidth]{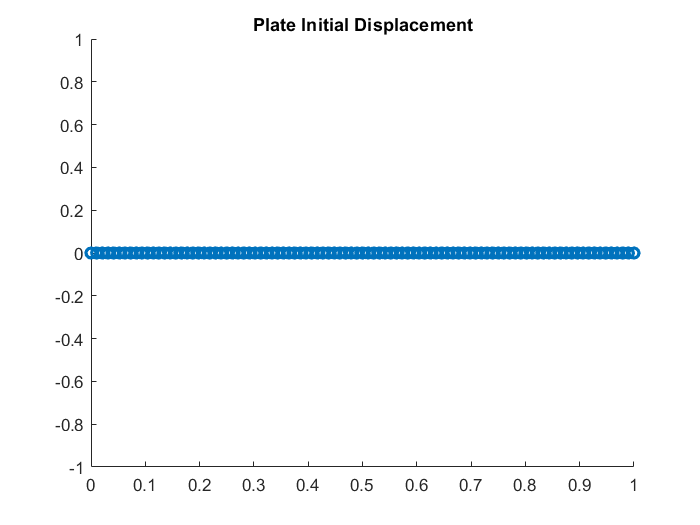}

\columnbreak

\includegraphics[width=0.85\linewidth]{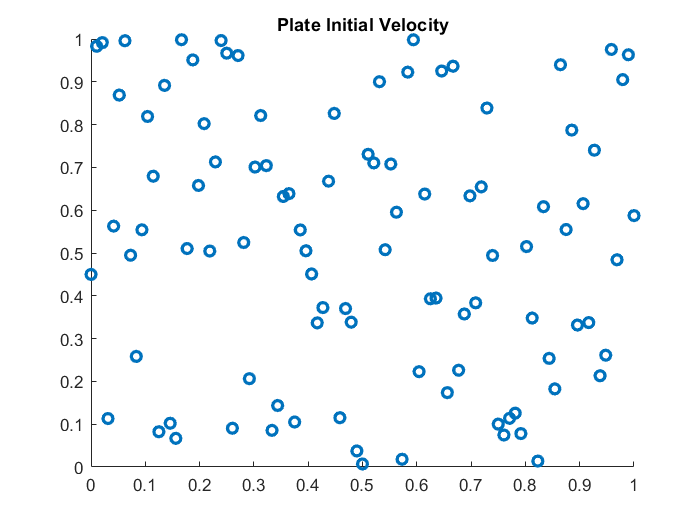}
\end{multicols}
\caption{Initial plate displacement and velocity}
\label{fig:plate-init}
\end{figure}

\begin{figure}
    \centering
    \begin{multicols}{2}
    \includegraphics[width=0.85\linewidth]{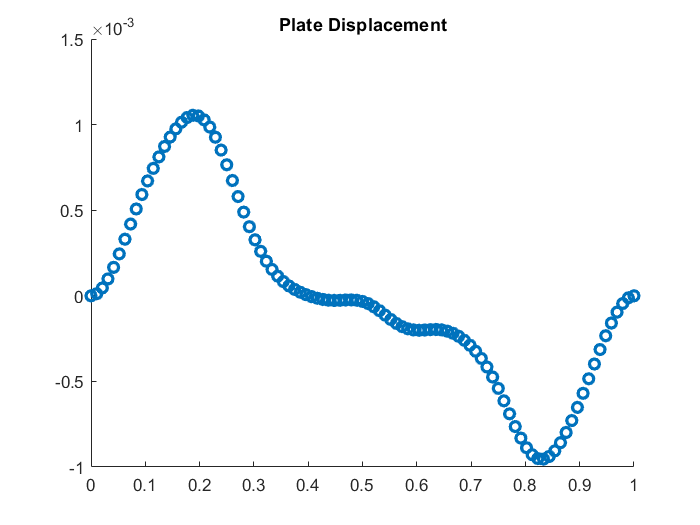}

    \columnbreak

    \includegraphics[width=0.85\linewidth]{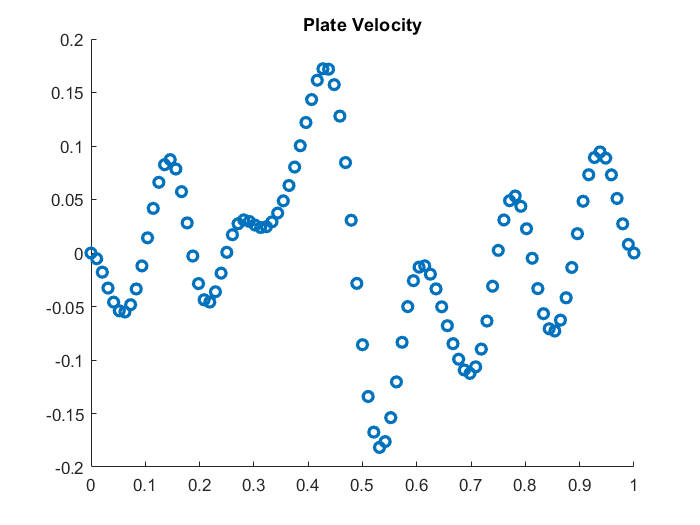}
    \end{multicols}

    \caption{Final plate displacement and velocity at $T = 0.001$}
    \label{fig:plate-final}
\end{figure}

\begin{figure}[h!]
\centering
\begin{multicols}{2}
\includegraphics[width=0.85\linewidth]{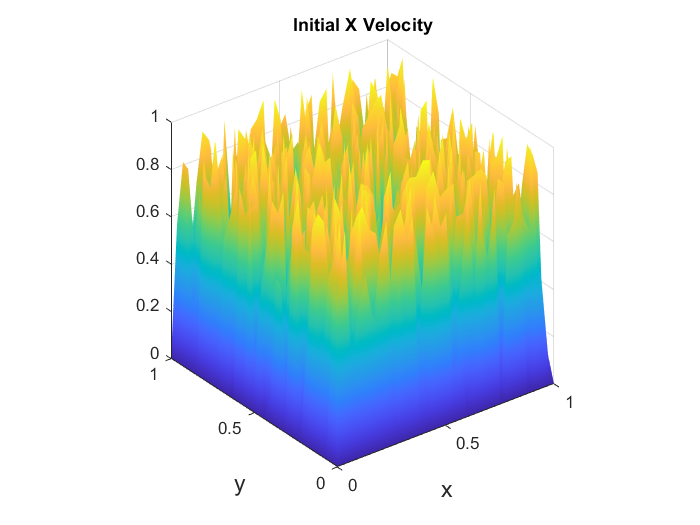}

\columnbreak

\includegraphics[width=0.85\linewidth]{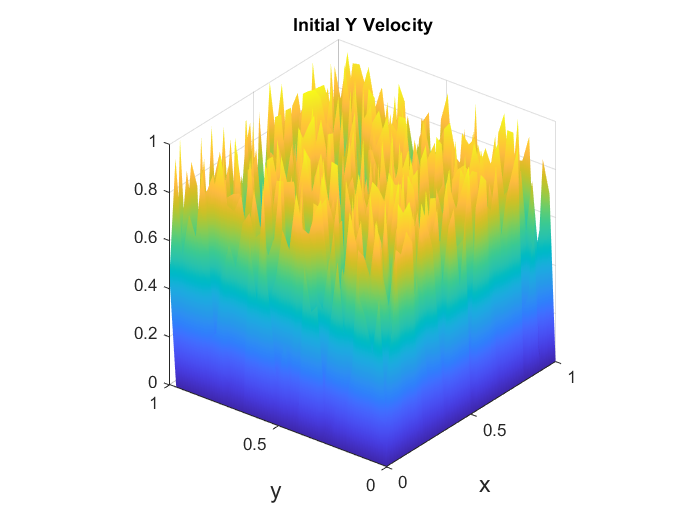}
\end{multicols}
\caption{Initial fluid velocities}
\label{fig:fluid-init}
\end{figure}

\begin{figure}[h]
    \centering
    \begin{multicols}{2}
    \includegraphics[width=0.85\linewidth]{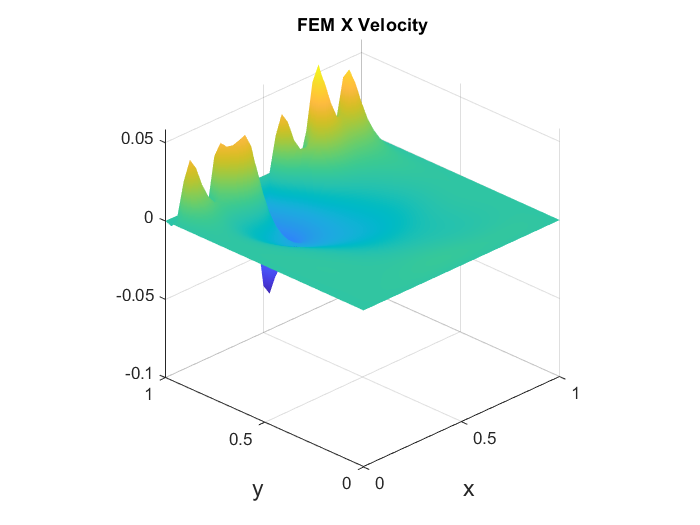}

    \columnbreak
    \includegraphics[width=0.85\linewidth]{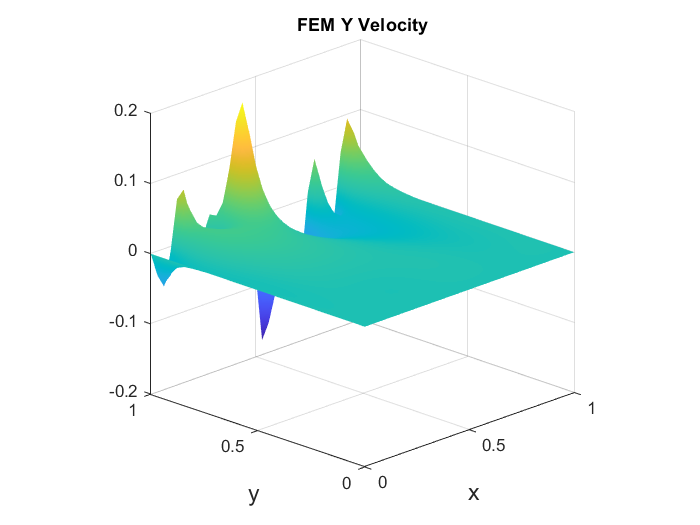}
    \end{multicols}

    \caption{Final fluid velocities at $T = 0.001$}
    \label{fig:fluid-final}
\end{figure}
\FloatBarrier

\section{Conclusion}
Theorem \ref{thm:mainResult} establishes a degree of smoothing on the system, in the sense of Gevrey. With this result in hand, drawing from the trajectories in structurally damped plates \cite{CT89}, \cite{CT90} and thermoelastic plates \cite{LT00}, \cite{LR95}, an appropriate notion of controllability for the FSI \eqref{FSI} is null controllability \cite{Z92}. This is the subject of an upcoming work. 

For now, the authors note that in contrast to thermoelastic plates, where all three dynamics evolve in the interior, the plate portion of the fluid structure system \eqref{FSI} evolves on the boundary. This makes the task of establishing controllability significantly more difficult as the trace operator is not appropriately bounded. There are also issues arising due to the presence of pressure. Partial results and a discussion in this direction are provided in the thesis \cite{M24}. 

\newpage
\bibliographystyle{spmpsci}
\bibliography{mybibliography.bib}

\end{document}